\documentclass[reqno,a4paper,10pt]{amsart}
\usepackage{amsmath,amsfonts,amssymb,amscd,amsthm,amsbsy,upref,mathrsfs}
\usepackage[T1]{fontenc}
\newtheorem{theorem}{Theorem}[section]
\newtheorem{lemma}[theorem]{Lemma}
\newtheorem{proposition}[theorem]{Proposition}
\newtheorem{corollary}[theorem]{Corollary}

\newtheorem{fact}[theorem]{Fact}
\newtheorem{claim}{Claim}

\theoremstyle{definition}
\newtheorem{definition}[theorem]{Definition}

\theoremstyle{remark}
\newtheorem{remark}[theorem]{Remark}

\newtheorem*{notation}{Notation}

\numberwithin{equation}{section}
\newcommand{\bl}{\begin{lemma}}
\newcommand{\el}{\end{lemma}}
\newcommand{\bfa}{\begin{fact}}
\newcommand{\efa}{\end{fact}}
\newcommand{\bpr}{\begin{proposition}}
\newcommand{\epr}{\end{proposition}}
\newcommand{\bp}{\begin{proof}}
\newcommand{\ep}{\end{proof}}
\newcommand{\bd}{\begin{definition}}
\newcommand{\ed}{\end{definition}}
\newcommand{\bt}{\begin{theorem}}
\newcommand{\et}{\end{theorem}}
\newcommand{\bc}{\begin{corollary}}
\newcommand{\ec}{\end{corollary}}
\newcommand{\bn}{\begin{notation}}
\newcommand{\en}{\end{notation}}
\newcommand{\br}{\begin{remark}}
\newcommand{\er}{\end{remark}}
\newcommand{\bcl}{\begin{claim}}
\newcommand{\ecl}{\end{claim}}

\newcommand{\vvert}[1][\cdot]{\vert #1\vert}
\newcommand{\N}{{\mathbb{N}}}

\newcommand{\al}{\alpha}
\newcommand{\be}{\beta}
\newcommand{\e}{\varepsilon}

\newcommand{\bnum}{\begin{enumerate}}
\newcommand{\enum}{\end{enumerate}}
\newcommand{\mc}{\mathcal}
\newcommand{\mt}{\mc{T}}

\newcommand{\fa}{f_{\alpha}}
\newcommand{\fb}{f_{\beta}}
\newcommand{\fg}{f_{\gamma}}

\numberwithin{subsection}{section}
\numberwithin{equation}{section}

\newcommand{\norm}[1][\cdot]{\lVert #1\rVert}
\DeclareMathOperator{\supp}{supp}
\DeclareMathOperator{\maxsupp}{maxsupp}
\DeclareMathOperator{\minsupp}{minsupp}
\DeclareMathOperator{\ran}{range}

\DeclareMathOperator{\suc}{succ}

\begin{document}
\title[Tightness in spaces with unconditional basis]{Types of tightness in spaces with unconditional basis}

\author{Antonis Manoussakis, Anna Pelczar-Barwacz}
\address[A. Manoussakis]{Department of Sciences, Technical University of Crete, GR 73100, Greece}
\email{amanousakis@isc.tuc.gr}

\address[A. Pelczar-Barwacz]{Institute of Mathematics, Faculty of Mathematics and Computer Science, Jagiellonian University, {\L}ojasiewicza 6, 30-348 Krak\'ow, Poland}
\email{anna.pelczar@im.uj.edu.pl}

\begin{abstract}
We present a reflexive Banach space with an unconditional basis which is quasi-minimal and tight by range, i.e. of type (4) in Ferenczi-Rosendal list within the framework of Gowers' classification program of Banach spaces, but contrary to the recently constructed space of type (4) also tight with constants, thus essentially extending the list of known examples in Gowers' program. The space is defined on the base on a boundedly modified mixed Tsirelson space with use of a special coding function.
\end{abstract}
\maketitle
\section*{Introduction}

The "loose" classification program for  Banach spaces was started by W.T. Gowers in the celebrated paper \cite{g2}. The goal is to identify classes of Banach spaces which are
\begin{itemize}
 \item hereditary, i.e. if a space belongs to a given class, then any its closed infinite dimensional subspace belongs to the same class,
 \item inevitable, i.e. any Banach space contains an infinite dimensional subspace in one of those classes,
\item defined in terms of the richness of the family of bounded operators on/in the space.
\end{itemize}
The program was inspired by the famous Gowers' dichotomy \cite{g1} exhibiting the first two classes: spaces with an unconditional basis and hereditary indecomposable spaces. Recall that a space is called \textit{hereditarily indecomposable} (HI) if none of its closed infinite dimensional subspaces is a direct sum of its two closed infinite dimensional subspaces.

The research now concentrates on identifying classes in terms of the family of isomorphisms defined in a space. The richness of this family can be stated in various "minimality" conditions, whereas the lack of certain type isomorphic embeddings of subspaces of a given space is described by different types of "tightness" of the considered space.  

Recall that a Banach space is \textit{minimal} if it embeds isomorphically into any its closed infinite dimensional subspace. Relaxing this notion one obtains \textit{quasi-minimality}, which asserts that any two infinitely dimensional subspaces of a given space contain further two isomorphic infinitely dimensional subspaces. Relaxing the notions of minimality or adding additional requirements of choice of isomorphic subspaces in quasi-minimality case leads to different types of minimality of a space, contrasted in \cite{g2,fr1} with different types of tightness, categorized in \cite{fr1}. 

Recall that a subspace $Y$ of a Banach space $X$ with a basis $(e_n)$ is \textit{tight in} $X$ if there is a sequence of successive subsets $I_1<I_2<\dots$ of $\N$ such that the support of any isomorphic copy of $Y$ in $X$ intersects all but finitely many $I_n$'s. $X$ is called \textit{tight} if any of its subspaces is tight in $X$. Adding requirements on the subsets $(I_n)$ with respect to given $Y$ one obtains more specific notions, in particular in \textit{tightness by support} the subsets  $(I_n)$ witnessing tightness of a subspace $Y$ spanned by a block sequence $(x_n)$ are chosen to be supports of $(x_n)$.  W.T. Gowers in \cite{g2} shows that every Banach space contains either a quasi-minimal subspace or a subspace tight by supports. The counterpart for minimality is tightness - by \cite{fr1} every Banach space contains contains a subspace which is either tight or minimal. 

A natural relaxing of the notion of tightness by supports, called \textit{tightness by range} assuring that one can choose subsets $(I_n)$ to be ranges of $(x_n)$ (recall that a range of a vector is the smallest interval containing the support this vector) has also its dichotomy counterpart in a stronger form of quasi-minimality, namely sequential minimality \cite{fr1}. A Banach space $X$ is \textit{sequentially minimal} if it is block saturated with block sequences $(x_n)$ with the following property: any subspace of $X$ contains a sequence equivalent to a subsequence of $(x_n)$. 

Finally, on the side of minimality type properties, one can relax the notion of minimality requiring that the considered space $X$ is only finitely represented in any its infinitely dimensional subspace. Such \textit{local minimality} is contrasted in a dichotomy in \cite{fr1} with \textit{tightness with constants} describing a strict control on the embedding constants - the sequence $(I_n)$ associated to a subspace $Y$ of $X$ has the following property: for any $K\in\N$ the subspace $Y$ does not embed with constant $K$ into $[e_i:\ i\not\in I_K]$.   

The obvious observations relate some of the properties listed above to HI/uncon\-di\-tio\-nal dichotomy - in particular clearly any HI space is quasi-minimal and any tight basis is unconditional. V. Ferenczi and C. Rosendal presented a list of classes within the framework of Gowers' classification program in \cite{fr1} and according to this list they examined in \cite{fr2} the spaces already known. The list of examples of the main classes was completed by the recent work of V. Ferenczi and Th. Schlumprecht \cite{fs}
and by the recent work of S.A. Argyros and the authors \cite{amp}.
We recall now the list of classes developed in \cite{fr1} as stated in \cite{fs}, mentioning also some already known examples.
\begin{theorem}[Ferenczi-Rosendal classification]
Any infinite dimensional Banach space contains a subspace with a basis from one of the following classes:
\begin{enumerate}
\item HI, tight by range (Gowers space with asymptotically unconditional basis of \cite{g} by \cite{fr2}),
\item HI, tight, sequentially minimal (version of Gowers-Maurey space by \cite{fs}),
\item tight by support (Gowers space with unconditional basis of \cite{g0} by \cite{fr2}),
\item with unconditional basis, tight by range, quasi-minimal (unconditional version of Gowers HI space with an asymptotically unconditional basis of \cite{g} by \cite{amp}),
\item with unconditional basis, tight, sequentially minimal (Tsirelson space by \cite{fr2}),
\item with unconditional basis, minimal ($\ell_p$, $c_0$, dual to Tsirelson space by \cite{cjt}, Schlum\-precht space by \cite{as}).
\end{enumerate}
\end{theorem}
The full Ferenczi-Rosendal list \cite[Theorem 8.4]{fr1} splits all of the above classes with respect to the dichotomy: local minimality versus tightness with constants. Further splitting of some of the above classes concerns the asymptotic structure of the space, i.e. the dichotomy: strong $\ell_p$-asymptoticity versus uniform inhomogeneity of \cite{tc}. 

The aim of the present paper is to examine the class (4) with respect to its local structure. We briefly sketch the proof of local minimality of the space $\mc{X}_{(4)}$ constructed in \cite{amp} and concentrate on the construction of an example on the other edge of the class (4). Namely we prove the following.
\begin{theorem}
There exists a reflexive space $\mathcal{X}_{cr}$ with an unconditional basis which is quasi-minimal, tight by range and with constants.
\end{theorem}
As it was mentioned above the example of \cite{amp} is an unconditional version of the Gowers HI space with an asymptotically unconditional basis \cite{g0}, using the standard framework of a space constructed on the basis of a mixed Tsirelson space, defined by a norming set closed under certain operations. The standard operations include taking averages of certain block sequences (as $(\mc{A}_n, \theta_n$-operations or $(\mc{S}_n,\theta_n)$-operations), projections on the subsets of $\N$ or intervals in $\N$, change of signs etc. Taking averages could be restricted to a special family of block sequences, picked usually by means of so-called "coding function"; this method, introduced by B. Maurey and H. Rosenthal, exploited by W.T. Gowers and B. Maurey lead to the construction of the first HI space \cite{gm}. 
 
In the case of the space $\mc{X}_{(4)}$ in  \cite{amp} its norming set is closed under change of signs, certain $(\mc{A}_{n_j},\frac{1}{m_j})$-operations, projections on intervals, and - in order to ensure tightness by range -  under the special "Gowers operation" used in \cite{g0}, i.e. scaled projection on $\mc{S}_1$ sets. This structure allows also for the usual reasoning on presence of finitely dimensional copies of $\ell_\infty$ in every subspace of $\mc{X}_{(4)}$ (cf. \cite{m}), thus also for local minimality. 

The typical way to provide tightness with constants in the considered space is to base its construction on the Schreier families instead of families $(\mc{A}_n)$. The strong asymptotic structure of the space, by \cite{fr1}, provides the desired type of tightness. However, using Schreier families in the definition of the norming set neutralizes the effect of "Gowers operation" used in the example of \cite{amp}. Therefore in order to construct the space $\mc{X}_{cr}$ with desired properties one needs other tools which would "spoil" the sequential minimality of regular modified mixed Tsirelson spaces defined by Schreier families proved in \cite{kmp}.

We present here a variant of a standard construction on the basis of boundedly modified mixed Tsirelson spaces \cite{adkm}. The norming set of the constructed Banach space $\mc{X}_{cr}$ is closed under change of signs, projection on intervals, and $(\mc{S}_{n_j},\frac{1}{m_j})$-operations on certain sequences, partly defined by a carefully chosen coding function. As it was mentioned earlier, tightness with constants is ensured by a strong asymptotic structure of the space, the quasi-minimality of the space follows by the regularity of the applied operations, whereas tightness by range follows by the use of the coding function. The key point in the choice of the coding function is its "complexity level" - high enough to spoil sequential minimality and ensure tightness by range, but still low enough to preserve the quasi-minimality of the built space. The presented construction exhibits large possibilities - within the framework of spaces built on the basis of mixed Tsirelson space - of designing the 
properties 
of the constructed space by means only of the coding function involved in the definition of the norming set.

We describe now the contents of the paper. We recall the standard notation in the first section. The second section is devoted to the definition and basic properties of our space $\mc{X}_{cr}$, while in the third section we prove the quasi-minimality and tightness properties of $\mc{X}_{cr}$. In the last section we sketch the proof of the local minimality of the space $\mc{X}_{(4)}$ of \cite{amp}. 

\section{Preliminaries}
We recall  the basic definitions and standard notation.

By a {\em  tree}  we shall mean a non-empty partially ordered  set $(\mt, \leq)$ for which the set $\{ y \in \mt:y \leq x \}$ is linearly ordered and finite for each $x \in \mt$. If $\mt' \subseteq \mt$ then we say that $(\mt',\leq)$ is a {\em subtree}  of $(\mt,\leq)$. The tree $\mt$ is called {\em finite}  if the set $\mt$ is finite. The \textit{root} is the smallest element of the tree (if it exists).   
A {\em branch}  in $\mt$ is a maximal linearly ordered set in $\mt$. The {\em immediate successors}  of $x \in \mt$, denoted by $\suc (x)$, are all  the nodes $y \in \mt$ such that $x < y$ but there is no $z \in \mt$ with $x < z < y$. If $X$ is a linear space, then a {\em tree in $X$}  is a tree whose nodes are vectors in $X$.

Let $X$ be a Banach space with a basis $(e_i)$. The \textit{support} of a vector $x=\sum_i x_i e_i$ is the set $\supp x =\{ i\in \N : x_i\neq 0\}$, the \textit{range} of $x$, denoted by $\ran x$ is the minimal interval containing $\supp x$. Given any $x=\sum_i a_ie_i$ and finite $E\subset\N$ put $Ex=\sum_{i\in E}a_ie_i$. We write $x<y$ for vectors $x,y\in X$, if $\max\supp x<\min \supp y$. A \textit{block sequence} is any sequence $(x_i)\subset X$ satisfying $x_{1}<x_{2}<\dots$. A closed subspace spanned by an infinite block sequence $(x_n)_{n\in\N}$ is called a \textit{block subspace} and denoted by  $[x_n: n\in\N]$. 

We shall consider two hierarchies of families of finite subsets of $\N$, namely families $(\mc{A}_n)_{n\in\N}$, defined as $\mc{A}_n=\{F\subset\N:\# F\leq n\}$ for each $n\in\N$, and \textit{Schreier families} $(\mc{S}_n)_{n\in\N}$, introduced in \cite{aa}, defined by induction:
\begin{align*}
\mc{S}_0 &=\{\{ k\}:\ k\in\N\}\cup\{\emptyset\}, \\
\mc{S}_{n+1}&  =\{F_1\cup\dots\cup F_k:\ k\leq F_1<\dots<F_k, \
f_1,\dots, f_k\in \mc{S}_n\}, \ \ n\in\N\,.
\end{align*}
We can also define modified Schreier families $(\mc{S}_n^M)_{n\in\N}$ by replacing in the definition above the condition "$F_1<\dots<F_k$" by "$F_1,\dots,F_k$ are pairwise disjoint". The following observation proves that these families coincide. 
\begin{lemma}\label{schreier} \cite[Lemma 1.2]{adkm}
 For any $n\in\N$ we have $\mc{S}_n=\mc{S}_n^M$.
\end{lemma}

Fix a family $\mc{M}$ of finite subsets of $\N$. We say that a sequence $(E_1,\dots, E_k)$ of subsets of $\N$ is
 \bnum
 \item $\mc{M}$-\textit{admissible}, if $E_1<\dots<E_k$ and $(\min E_i)_{i=1}^k\in\mc{M}$,
 \item $\mc{M}$-\textit{allowable}, if $(E_i)_{i=1}^k$ are pairwise disjoint and $(\min E_i)_{i=1}^k\in\mc{M}$.
 \enum
Let $X$ be a Banach space with a basis. We say that a sequence $x_1,\dots ,x_n$ is $\mc{M}$-\textit{admissible} (resp. \textit{allowable}), if $(\supp x_i)_{i=1}^n$ is $\mc{M}$-admissible (resp. allowable).

A $(\mc{M},\theta)$-operation, with $0<\theta\leq 1$, is an operation associating with a sequence $(x_1,\dots,x_k)$ with $(\minsupp x_i)_{i=1}^k\in\mc{M}$ the vector $\theta(x_1+\dots+x_k)$.

Fix sequences $(\theta_n)_n\subset (0,1)$, $(k_{n})\nearrow+\infty$ and $(\mc{M}_n)$ with either $\mc{M}_n=\mc{A}_{k_n}$ for all $n$ or $\mc{M}_n=\mc{S}_{k_n}$ for all $n$. A mixed Tsirelson space $T[(\mc{M}_n, \theta_n)_n]$ is defined to be the completion of $c_{00}(\N)$ endowed with the norm, whose norming set $K$ is the smallest subset of $c_{00}(\N)$ containing $(\pm e_n)_n$, where $(e_n)_n$ is the canonical basis of $c_{00}(\N)$, and closed under all $(\mc{M}_n,\theta_n,)$-operations on block sequences. If one allows also $(\mc{M}_n,\theta_n)$-operations on sequences of vectors with pairwise disjoint supports for some $n\in\N$, one gets (boundedly) modified mixed Tsirelson spaces, cf. \cite{adm}. 

The first famous member of the family of spaces $T[(\mc{A}_{k_n}, \theta_n)_n]$ is Schlum\-precht space \cite{s}, the first space known to be arbitrarily distortable, see also  \cite{m} for a study of this class of spaces. Spaces $T[(\mc{S}_{k_n},\theta_n)_n]$ were introduced in \cite{ad}. Allowing some $(\mc{M}_n,\theta_n)$-operations on special block sequences, defined by means of a suitably chosen coding function, opened the gate to Gowers-Maurey construction of the first known HI space \cite{gm}. Adding in the definition of the norming set other operations of different kind allowed for building spaces enjoying extreme properties, as HI asymptotically unconditional space of Gowers \cite{g} or quasi-minimal and tight by range space with unconditional basis \cite{amp}. 
\section{Definition of the space $\mc{X}_{cr}$}
The space we shall define is constructed on the basis of boundedly mo\-di\-fied mixed Tsirelson space $T_M[(\mc{S}_{n_j},\frac{1}{m_j})_j]$ with use of an additional coding function. First we describe the basic ingredients of the construction. 

We fix two sequences of natural numbers $(m_{j})_{j}$ and $(n_{j})_{j}$ defined
recursively as follows. We set $m_{1} =2$ and $m_{j+1}=m_{j}^{5}$ and $n_{1}=4$
and $n_{j+1}=15s_{j}n_{j}$ where $s_{j}=\log_{2}(m_{j+1}^{3})$, $j\geq 1$. 

Let $X$ be a Banach space with a basis $(e_i)$ satisfying the following:
\begin{equation}\label{standard}
\frac{1}{m_{2j}}\sum_i\norm[E_ix]\leq \norm[x] \text{ for any }x\in X,\ j\in\N, \ (E_i) - \mc{S}_{n_{2j}}\text{-admissible }
\end{equation}
We recall now standard facts on vectors of a special type in such a space.
\begin{definition}[Special convex combination] 
Fix $\e>0$ and $n\in\N$. 

We call a vector $y=\sum_{i\in F}a_ie_i$ an $(n,\e)$-\textit{basic special convex combination} (basic scc), if $F\in\mc{S}_n$ and scalars $(a_i)\subset [0,1]$ satisfy $\sum_{i\in F}a_i=1$ and $\sum_{i\in G}a_i<\e$ for any $G\in\mc{S}_{n-1}$. 

We call a vector $x=\sum_{i\in F}a_ix_i$ an $(n,\e)$-\textit{special convex combination} (scc) of $(x_i)$, if the vector $y=\sum_{i\in F}a_ie_{\minsupp x_{i}}$ is an $(n,\e)$-basic scc. 

We call a scc $x=\sum_{i\in F}a_ix_i$ of a normalized block sequence $(x_i)$ a \textit{seminormalized scc}, if $\norm[x]\geq 1/2$.
\end{definition}
It is well known, see \cite[Prop. 2.3]{Artol}, that for every $n\in\N$,
$\e>0$ and every $L\subset\N$ there exists an basic $(n,\e)$-scc,
$x=\sum_{n\in F}a_{n}e_{n}$ such that $F$ is a maximal $\mc{S}_{n}$-subset of $L$.   The next lemma provides seminormalized
 scc's in every block subspace.
\begin{lemma}\label{f1} \cite[Lemma 4.5]{adm} For every $n\in \N$, $\e>0$ there is $l(n,\e)\in\N$ such that for any block sequence $(x_i)$  there is $F\in\mc{S}_{l(n,\e)}$ such that there is an $(n,\e)$-scc $x$ supported on $(x_i)_{i\in F}$ with $\norm[x]\geq 1/2$.
\end{lemma}
Recall for any $n\in\N$ and $\e>0$  the constant $l(n,\e)\geq n$ depends only on sequences $(n_j)$ and $(m_j)$. 

Given any $n\in\N$ let $\rho(n)$ be the constant $l(n_{2s},m_{2s}^{-2})$ obtained by the above Lemma, where $s\in\N$ is minimal with 
\begin{equation}\label{rho}
n^2\leq m_{2s}.
\end{equation}
We fix a partition of $\N$ into two infinite sets $L_{1},L_{2}$. Let
$$
\mc{G}=\{(E_{1},\dots,E_{n}):\ E_{1}<E_{2}<\dots<E_{n} \text{ intervals of }\N,\
n\in\N\}
$$
and take a  $1-1$ \textit{coding function} $\sigma:\mc{G}\to 2L_{2}$ such that for any sequence
$(E_{1},\dots, E_{n})\in\mc{G}$, $n\geq 2$, we have
\begin{equation}
  \label{eq:2}
n_{\sigma(E_{1},\dots, E_{n})}\geq \rho(\max E_{n})+\max E_n\, .
\end{equation}
Let $W$ be the smallest subset of $c_{00}(\N)$ such that
\begin{enumerate}
\item[($\alpha$)] $(\pm e_n)_n\in W$, where $(e_n)_n$ is  the canonical basis of $c_{00}(\N)$
\item[($\beta$)] for any $f\in W$ and $g\in c_{00}(\N)$ with $\vvert[f]=\vvert[g]$ also $g\in W$,
\item[($\gamma$)] $W$ is closed  under the
$(\mc{S}_{n_{2j}},m_{2j}^{-1})$-operations on any allowable sequences,
\item[($\delta$)] $W$ is closed under the 
$(\mc{S}_{n_{2j+1}},m_{2j+1}^{-1})$-operations on $(2j+1)$-depen\-dent sequences.
\end{enumerate}
In order to compete the definition we need to define dependent sequences.
\begin{definition}[Dependent sequence]\label{special}
A block sequence $(f_{i})_{i\in F}\subset c_{00}(\N)$ is called a $(2j+1)$-\textit{dependent sequence} if $(f_{i})_{i\in F}$ is $\mc{S}_{n_{2j+1}}$-admissible, each $f_{i}$ is of the form  $f_{i}=m_{2j_{i}}^{-1}\sum_{k\in K_i}f_{i,k}$ and for some sequences $(E_r)_{r\in A}\in \mc{G}$, with the index set $A\subset\N$ represented as a sum $\cup\{A_{k}: k\in K_i, i\in F\}$ of intervals, the following hold
\begin{enumerate}
  \item $w(f_{1})=m^{-1}_{2j_{1}}$, $j_{1}\in L_{1}$ and $m_{2j_1}>n_{2j+1}$,
   \item $2j_{i+1}=\sigma(E_j:\ j\in A_{k},\ k\in K_l,\ l\leq i)$ for any $i<\max F$,
  \item $\supp f_{i,k}\subset \cup_{j\in A_{k}}E_{j}$ for any $k\in K_i$, $i\in F$,
  \item $A_{\min K_i}$ is a singleton for each $i\in F$, and 
$(E_r)_{r\in A_{k}}$ is  $\mc{S}_{\rho(\max E_{\max  A_{k-1}})}$-admissible for any $k\in K_i$, $k>\min K_i$, $i\in F$.
  \end{enumerate}
Any functional of the form $m_{2j+1}^{-1}\sum_{i\in F}f_i$, where $(f_i)_{i\in F}$ is a dependent sequence, is called a \textit{special functional}.
\end{definition}
Notice that as $(E_r)_r\in\mc{G}$, by (3) each family $(f_{i,k})_{k\in K_i}$ is $\mc{S}_{n_{2j_i}}$-admissible.

Let $\mc{X}_{cr}$ be the completion of $c_{00}(\N)$ with the norm $\norm[\cdot]$ defined by $W$ as its norming set, i.e. $\norm[\cdot]=\sup\{|f(\cdot)|:\ f\in W\}$. 

\begin{remark}\label{rem-k}
a) The canonical basis $(e_n)_{n}$ of  $\mc{X}_{cr}$ is 1-sign unconditional by $(\beta)$.

b) Note that in the definition of the dependent sequences the 
admissibility of the functional
chosen in the $(i+1)$-th step depends not on 
the supports or ranges of previously chosen functionals, but on the choice of
some intervals containing ranges of the previously chosen functionals.

c)  By (3)  of the definition of the special functionals it follows that we can choose $f_{i,k}=0$.
Also by (3) any restrictrion of a special functionals to a subset of $\N$ is also a special functional. 
This property easily implies also that the set $W$ is closed under the projections on subsets of $\N$ and 
$(e_n)_{n}$ is a 1-unconditional basis.

d) The space $\mc{X}_{cr}$ satisfies \eqref{standard} by $(\gamma)$.

e) Reflexivity of $\mc{X}_{cr}$ can be proved by repeating the argument of \cite{ad}.

f) The norming set $W$ of $\mc{X}_{cr}$ is contained in the norming set of the 
modified mixed Tsirelson space 
$T_M[(\mc{S}_{n_j},\frac{1}{m_j})_j]$, cf. \cite{adm}. 
\end{remark}
As in the previous cases of norming sets defined to be closed under certain operations every functional $f\in W$ admits a tree-analysis which in the present case is described as follows.
\begin{definition}[Tree-analysis of a functional]
  \label{treeanal}
Let $f\in W$. A family $(\fa)_{\al\in \mc{T}}$, where $\mc{T}$ is a rooted finite tree is a tree-analysis of $f$ if the following are satisfied
 \begin{enumerate}
 \item $f=f_{0}$ where $0$ denotes the root of $\mc{T}$.
\item If $\al$ is a maximal element of $\mc{T}$ then $\fa=\pm e_{n}^{*}$ for some $n\in\N$.\\
If $\al\in\mc{T}$ is not maximal, then one of the following conditions holds
\item $\fa=\frac{1}{m_j}\sum_{\beta\in \suc (\al)}f_\beta$ with $(f_\beta)_{\beta\in\suc (\al)}$ $\mc{S}_{n_j}$-allowable, for $j\in 2\N$,
\item $\fa=\frac{1}{m_j}\sum_{\beta\in \suc (\al)}f_\beta$ with $(f_\beta)_{\beta\in\suc (\al)}$ $\mc{S}_{n_j}$-admissible, for $j\in 2\N+1$. \\
In the above two cases we set the weight $w(\fa)$ of $\fa$ as $w(\fa)=m_{j}^{-1}$.
 \end{enumerate}
For any $0\neq\al\in\mt$ we set $tag(\al)=\prod_{\beta\prec\al} w(\fb)$ and $ord(\al)$ to be equal to the length of the branch linking $\al$ and the root $0$. 
\end{definition}
\begin{lemma}\cite[Lemma  4.6]{adm}  \label{admi} Let $j\in\N$, $f\in W$ be a norming functional with a tree-analysis $(\fa)_{\al\in\mt}$. Let
$$
\mc{F}=\{\al\in\mt: \prod_{\beta\prec\al}w(\fb)>\frac{1}{m_j^{2}}\,\text{ and }\, w(\fb)\geq \frac{1}{m_{j-1}}\,\text{ for all }\,\beta\prec\alpha\}.
$$
Then for any subset  $\mc{G}$ of $\mc{F}$ of incomparable nodes the set $\{\fa:\al\in \mc{G}\}$ is $\mc{S}_{\frac{1}{5}n_j}$-allowable and for any $\al\in\mc{F}$ we have $ord(\al)\leq m_j$.
\end{lemma}
\begin{proof}
For every $\al\in\mc{G}$  by the assumptions we get
$$
\frac{1}{m_{j}^{2}}<\prod_{\beta\prec\al}w(\fb)\leq(\frac{1}{m_{1}})^{ord(\al)} \Rightarrow ord(\al)\leq 2\log_{m_{1}}(m_{j}).
$$
Since for all $\beta\prec\al$ , $n_{\beta}\leq  n_{j-1}$ it
follows  for every $s\leq 2\log_{m_{1}}(m_{j})$ the nodes  of $\mc{G}$
on the $s$-th level of the tree  are at most $\mc{S}_{sn_{j-1}}$-allowable. It
follows that the nodes of $\mc{G}$ are at most
$\mc{S}_{2\log_{m_{1}}(m_{j})n_{j-1}} $-allowable, thus also $\mc{S}_{\frac{1}{5}n_{j}}$-allowable.
\end{proof}
\begin{lemma}\label{1.4}\cite{adm}
For any $(n_j,m_j^{-2})$-scc $x=\sum_{i\in F}a_{i}x_{i}\in\mc{X}_{cr}$
with $\norm[x_{i}]\leq C$ for every $i\in F$ and any
$\mc{S}_{n_j-1}$-allowable family  $(f_p)_{p\in A}\subset W$ of norming
functionals we have
$$
\sum_{p\in A}f_{p}
(x)
\leq  3C.
$$
\end{lemma}
\begin{definition}  \label{dRIS} 
Fix $C>0$. A block sequence $(x_{k})_{k}$ is called a $C$-\textit{rapidly increasing sequence} ($C$-RIS) if $\norm[x_{k}]\leq  C$ for each $k$ and there exists a strictly increasing sequence $(j_k)\subset\N$ such that
  \begin{enumerate}
  \item $\maxsupp x_k\leq \frac{m_{j_{k+1}}}{m_{j_k}}$ for any $k$,
 \item $\vvert[f(x_{k})]\leq Cw(f)$ for every $f\in W$ with $w(f)>m_{j_{k}}^{-1}$ and any $k$.
\end{enumerate}
\end{definition}
By repeating the proof of \cite[Prop. 4.12]{adm} we obtain the following
\begin{lemma}\label{f3}
For any $(n_j,m_j^{-2})$-scc $x=\sum_ka_kx_k$ of a $C$-RIS $(x_k)$ defined by a sequence $(j_k)$ with $j+2<j_1$ and any norming functional $f\in W$ with weight $w(f)=m_s^{-1}$ we have 
$$
|f(x)|\leq 
\begin{cases}
\frac{14C}{m_sm_j}& \text{ if }s<j\\
\frac{8C}{m_j}& \text{ if }s=j\\
\frac{8C}{m_j^2}& \text{ if }s>j
\end{cases}
.$$
In particular $\norm[x]\leq \frac{8C}{m_j}$ and for any $\mc{S}_{n_{2s}}$-allowable family $(\fa)_{\al\in A}\subset W$  with $2s<j$ we have
\begin{equation}\label{f3a}
\sum_{\al\in A}\fa(m_j x)\leq 14C\,.
\end{equation}
\end{lemma}
Notice that by the above Lemma a sequence of scaled vectors $(m_{j_k}x_k)$, where each $x_k$ is a $(n_{j_k}, m_{j_k}^{-2})$-scc of some $C$-RIS, satisfying (1) of Def. \ref{dRIS}, is also a 14$C$-RIS. As by Lemma \ref{f1} any block subspace of $\mc{X}_{cr}$ contains a 2-RIS of seminormalized scc's, by Lemma \ref{f3} any block subspace contains also for any $j\in\N$ a scaled $(n_{2j+1},m_{2j+1}^{-2})$-scc of 28-RIS.

By repeating the proof of  \cite[Lemma 4.10]{adm} with use of the above estimation we obtain the following.
\begin{lemma}\label{f4} 
Let $j>5$, $u=m_{2j+1}\sum_ka_kx_k$ be a scaled $(n_{2j+1}, m_{2j+1}^{-2})$-scc
of a 28-RIS $(x_k)$ defined by a sequence $(j_k)$ with $j+2<j_1$. Then any norming functional $f$ with a tree-analysis
$(\fa)_{\al\in\mt}$ such that $w(\fa)>\frac{1}{m_{2j+1}}$ for any $\al\in\mt$ satisfies 
$$
f(u)\leq \frac{1}{m_{2j}}\,.
$$
\end{lemma}
\section{Properties of the space $\mc{X}_{cr}$}
In this section we study the minimality properties of $\mc{X}_{cr}$. We deal first with tightness with constants as it follows immediately by \cite{fr1}.
\begin{definition}\cite{fr1}
 A Banach space $X$ with a basis $(e_n)$ is called \textit{tight with constants}, if for any infinite dimensional subspace $Y$ of $X$ there is a sequence of successive intervals $I_1<I_2<\dots$ such for any $K\in\N$ the subspace $Y$ does not embed with constant $K$ into $[e_i:\ i\not\in I_K]$. 
\end{definition}
Recall that a Banach space with a basis is  \textit{$\ell_1$-strongly asymptotic} if any $\mc{S}_1$-allowable sequence of normalized vectors $(x_1,\dots,x_n)$ is $C$-equivalent to the u.v.b. of $\ell_1^n$, for any $n\in\N$ and some universal $C\geq 1$. By $(\gamma)$ in the definition of its norming set and Remark \ref{rem-k}, the space $\mc{X}_{cr}$ is $\ell_1$-strongly asymptotic. Since $\mc{X}_{cr}$ is also reflexive, by \cite[Prop. 4.2]{fr1} we obtain the following.
\begin{theorem}
The space $\mc{X}_{cr}$ is tight with constants.  
\end{theorem}
We pass now to the proof of quasi-minimality, adapting it suitably we shall show also tightness by range. Recall that a Banach space is \textit{quasi-minimal}, if any two infinitely dimensional subspaces have further two infinitely dimensional subspaces which are isomorphic. 
\begin{theorem}
 The space $\mc{X}_{cr}$ is quasi-minimal.
\end{theorem}
\begin{proof}
Given two block subspaces $Y,Z$ of $X$ we pick a block RIS $(u_n)$ and $(v_n)$ satisfying the following. 
\bnum
\item[(A)] $u_n=m_{2j_n+1}\sum_{i\in I_n}b_iy_i$ is a scaled $(n_{2j_n+1}, m_{2j_n+1}^{-2})$-scc of 28-RIS $(y_i)$,\\
 $v_n=m_{2j_n+1}\sum_{i\in I_n}b_iz_i$ is a scaled $(n_{2j_n+1}, m_{2j_n+1}^{-2})$-scc of 28-RIS $(z_i)$,\\ 
$\minsupp u_n,\minsupp v_n> m_{2j_n+1}$ for all $n\in\N$, \\
$\sum_nm_{2j_n}^{-1}<100^{-1}$, 
\item[(B)] $y_i=m_{2j_i}\sum_{k\in K_i}b_{i,k}y_{i,k}$ is a scaled $(n_{2j_i},m_{2j_i}^{-2})$-scc of 2-RIS $(y_{i,k})_{k\in K_i}$,\\ 
$z_i=m_{2j_i}\sum_{k\in K_i}b_{i,k}z_{i,k}$ is a scaled $(n_{2j_i},m_{2j_i}^{-2})$-scc of 2-RIS $(z_{i,k})_{k\in K_i}$,\\  
$2j_i>2j_n+3$, $y_i^*=\frac{1}{m_{2j_i}}\sum_{k\in K_i}y_{i,k}^*$,  $z_i^*=\frac{1}{m_{2j_i}}\sum_{k\in K_i}z_{i,k}^*$, for each $i\in I_n$, $n\in\N$,
\item[(C)] $y_{i,k}$ is a $(n_{2j_{i,k}}, m_{2j_{i,k}}^{-2})$-scc with $\norm[y_{i,k}]\geq 1/2$ supported on some block sequence equivalent to a block sequence in $Y$, \\
$z_{i,k}$ is a $(n_{2j_{i,k}}, m_{2j_{i,k}}^{-2})$-scc with
$\norm[z_{i,k}]\geq 1/2$ supported on some block sequence equivalent
to a block sequence in $Z$,\\
 $y_{i,k}^*(y_{i,k})=1=z_{i,k}^*(z_{i,k})$, $\ran y_{i,k}^*=\ran
 y_{i,k}=\ran z_{i,k}^*=\ran z_{i,k}$ 
for each $k\in K_i$, $i\in I_n$, $n\in\N$,
\item[(D)] $(y_i^*)_{i\in I_n}$, $(z_i^*)_{i\in I_n}$ are $(2j_n+1)$-dependent
sequences, defined by the collection of intervals $(\ran y_{i,k})_{k\in K_i,i\in
I_n}$, $n\in\N$ (in the notation of Def. \ref{special} for each $k\in K_i$,
$i\in I_n$ we take as the collection $(E_j)_{j\in A_k}$ just one interval $\ran
y_{i,k}$).
\enum
The construction is straightforward - using Lemma \ref{f1} we construct first two infinite 2-RIS $(\hat{y}_k)_{k\in\N}\subset Y$ and $(\hat{z}_k)_{k\in\N}\subset Z$ of scc with norm at least 1/2 with a common sequence $(j_k)_k$. Allowing small perturbation we can also assume that $\ran \hat{y}_k=\ran\hat{z}_k$. Then we define $(y_i)$ and $(z_i)$ as 28-RIS of scaled scc's of $(\hat{y}_k)_k$ and $(\hat{z}_k)_k$ respectively, with the same coefficients with respect to $(y_{i,k})_{k\in K_i}\subset (\hat{y}_k)_k$ and $(z_{i,k})_{k\in K_i}\subset (\hat{z}_k)_k$. We repeat the procedure, building sequences of scaled scc's $(u_n)_n$ and $(v_n)_n$ on $(y_i)_i$ and $(z_i)_i$ respectively.

Note that by the definition of the coding function \eqref{eq:2}, using that $\max E_{i,\max K_{i}}=\maxsupp y_{i}$ we get
for each $i\in I_n, n\in\N$,
\begin{equation}
n_{2j_{i+1}}>\rho(\maxsupp y_i)+\maxsupp y_i>\rho(\maxsupp y_i)+3n_{2j_{n}+1}\,\,\,\tag{E}.
\end{equation}
We claim that the sequences $(u_n)$ and $(v_n)$ are equivalent. Take any non-negative scalars $(a_n)$ with $\norm[\sum_na_nu_n]=1$, let $u=\sum_na_nu_n$ and take a norming functional $f$ with a tree-analysis $(f_\al)_{\al\in\mt}$ and such that $f(u)=1$. Since the norming set $W$ is invariant under changing signs of coefficients by the condition ($\beta$) in the definition of the norming set $W$, we can assume that all coefficients of vectors $(u_n)$, $(v_n)$ and functional $f$ are non-negative.

By modifying the tree-analysis of $f$ we shall construct a tree-analysis of some norming functional $g$ such that $g(\sum_na_nz_n)\geq 1/6$. We shall make first some reductions, erasing nodes of $\mt$ with some controllable error. After the reduction we define a suitable replacements of certain nodes $f_\al$, $\al\in\mt$, in order to define $g$. 

First we introduce some notation. For any collection $\mc{E}$ of nodes of $\mt$ we shall write $\supp \mc{E}=\cup_{\al\in\mc{E}}\supp f_\al$. We shall prove several reductions, enabling us to restrict the tree-analysis of $f$ to the nodes convenient for the replacement procedure.

\

\textbf{1st reduction}. For any $n\in\N$ let 
\begin{align*}
\mc{P}_n=\{\al\in\mt: \ &\supp f_\al\cap\ran u_n\neq\emptyset\text{ and}\\
&\text{$\al\in\mt$  is minimal  with }w(f_\al)\leq m_{2j_n+1}^{-1}\}
\end{align*}
With error $2m_{2j_n}^{-1}$ we can assume that $(f_\al|_{\ran u_n})_{\al\in
\mc{P}_n}$ is $\mc{S}_{\frac{1}{5}n_{2j_n+1}}$-allowable and $\supp u_n\subset
\supp \mc{P}_n$ and for any $\al\in \mc{P}_n$ we have $ord(\al)\leq
m_{2j_n+1}$. 

\textit{Proof}. Notice first that by Lemma \ref{f4}  we have that 
$$
(f-f|_{\supp\mc{P}_n})(u_n)\leq \frac{1}{m_{2j_{n}}},
$$
hence with error $m_{2j_n}^{-1}$ we can assume that $\supp u_n\subset\supp
\mc{P}_n$. Now let
$$
\mc{P}_{n,1}=\{\al\in \mc{P}_n: \prod_{\beta\prec\alpha}w(\fb)\leq m_{2j_n+1}^{-2}\}.
$$
For every $\al\in \mc{P}_{n,1}$ choose $\be_{\al}\prec\al$ such that
\begin{equation*}
  \label{eq:6}
 \prod_{\gamma\prec\beta_{\al}}w(\fg)>m_{2j_n+1}^{-2}\,\,\textrm{and}\,\,
  \prod_{\gamma\preceq\be_{\al}}w(\fg)\leq m_{2j_n+1}^{-2}\,.
\end{equation*}
It follows that
\begin{equation}
  \label{eq:7}
  \frac{1}{m_{2j_n+1}^{2}}<\prod_{\gamma\prec\be_{\al}}w(\fg)\leq
w(f_{\be_{\al}})^{-1}\prod_{\gamma\preceq\be_{\al}}w(\fg)\leq \frac{m_{2j_n}}{m_{2j_n+1}^{2}}\,.
\end{equation}
Note that if $\al,\al_{1}\in \mc{P}_{n,1}$ the nodes  $\be_{\al},\be_{\al_{1}}$ are either incomparable or equal. Set
\begin{align*}
  \mc{R}_{n}=\{\beta_{\al}:\ \al\in \mc{P}_{n,1}\}.
\end{align*}
By Lemma \ref{admi} we get that
\begin{equation}
  \label{eq:33}
  \{f_{\be}:\beta\in \mc{R}_{n}\}\,\textrm{ is }\mc{S}_{n_{2j_n+1}-1 }\text{-allowable}.
\end{equation}
Consequently, using \eqref{eq:7}, \eqref{eq:33} and Lemma~\ref{1.4} for the scc $\sum_ib_iy_i$, we obtain 
\begin{align*}
  \label{eq:8}
 f|_{\supp \mc{P}_{n,1}}(m_{2j_n+1}\sum_{i\in I_n}b_{i}y_{i})
  &  \leq  m_{2j_n+1}\sum_{\beta\in \mc{R}_n} \left(\prod_{\gamma\prec\be_{\al}}w(\fg)\right)f_{\be}(\sum_{i\in I_n}b_{i}y_{i}) \\
  &  \leq  \frac{m_{2j_n}}{m_{2j_n+1}}  \sum_{\beta\in \mc{R}_n}f_{\be}(\sum_{i\in I_n}b_{i}y_{i})\\
 &  \leq 3\cdot 28\frac{1}{m_{2j_n}^2}\leq\frac{1}{m_{2j_{n}}}\,. 
\end{align*}
Now Lemma ~\ref{admi} applied to the family $\{\fa:\al\in \mc{P}_n\setminus\mc{P}_{n,1}\}$ finishes the proof of the 1st reduction.

\

\textbf{2nd reduction}. For any $n\in\N$ with error $m_{2j_n}^{-1}$ we can
assume that $w(f_\al)=m_{2j_n+1}^{-1}$ for any $\al\in \mc{P}_{n}$.

\textit{Proof}. Set $\mc{P}_{n,2}=\{\al\in \mc{P}_{n}: \ w(f_\al)< m_{2j_n+1}^{-1}\}$. For any $\al\in \mc{P}_{n,2}$ pick $i_\al\in I_n$ with $m_{2j_{i_\al}}^{-1}\geq w(f_\al)>m_{2j_{i_\al+1}}^{-1}$. If $w(f_\al)\leq m_{2j_i}^{-1}$ for any $i\in I_n$ let $i_\al=\max I_n$, if $w(f_\al)> m_{2j_i}^{-1}$ for all $i\in I_n$, put $i_\al=0$.

Split $f|_{\supp \mc{P}_{n,2}}(u_n)$ in the following way
\begin{align*}
f|_{\supp \mc{P}_{n,2}}(u_n)&\leq m_{2j_n+1}\sum_{i\in I_n}b_i\sum_{\al\in\mc{P}_{n,2}:i_\al>i}f_\al(y_i)\\
&+m_{2j_n+1}\sum_{i\in I_n}b_i\sum_{\al\in\mc{P}_{n,2}:i-2\leq i_\al\leq i} tag(\fa)f_\al(y_i)\\
&+m_{2j_n+1}\sum_{i\in I_n}b_i\sum_{\al\in\mc{P}_{n,2}:i_\al< i-2}f_\al(y_i)
\end{align*}
Fix $i\in I_n$ and compute, using for the last estimate the condition (1) of Def. \ref{dRIS}
\begin{align*}
\sum_{\al\in\mc{P}_n:i_\al>i}f_\al(y_i)&\leq \sum_{\al\in\mc{P}_n:i_\al>i}w(f_\al)\sum_{\gamma\succ\al}f_\gamma (y_i)\\
& \leq \sum_{\al\in\mc{P}_n:i_\al>i}\frac{1}{m_{2j_{i_\al}}}\sum_{\gamma\succ\al}f_\gamma (y_i)\\
& \leq \sum_{\al\in\mc{P}_n:i_\al>i}\frac{1}{m_{2j_{i+1}}}\sum_{\gamma\succ\al}f_\gamma (y_i)\\
&\leq \frac{1}{m_{2j_{i+1}}}28\maxsupp y_i\\
& \leq \frac{28}{m_{2j_{i}}}\,.
\end{align*}
For the second part notice that for each $\al\in\mc{P}_n$ there are at most 3 $i$'s in $I_n$ with $i_\al\in \{i-2,i-1,i\}$. Denote the set of all such $i$'s by $J_n$. As by the 1st reduction $(f_\al|_{\ran u_n})_{\al\in \mc{P}_n}$ is $\mc{S}_{\frac{1}{5}n_{2j_n+1}}$, using that $u_n$ is a scaled $(n_{2j_n+1}, m_{2j_n+1}^{-2})$-scc we obtain
\begin{align*}
 m_{2j_n+1}\sum_{i\in I_n}b_i\sum_{\al\in\mc{P}_{n,2}:i-2\leq i_\al\leq i} tag(\fa)f_\al(y_i)\leq \norm[m_{2j_n+1}\sum_{i\in\mc{J}_n}b_iy_i]\leq \frac{6\cdot 28}{m_{2j_n+1}}
\end{align*}
For the third part notice that if $i_\al\leq i-3$, then $w(\fa)\geq m_{2j_{i-2}}^{-1}$. Fix again $i\in I_n$, and estimate by definition of $\mc{P}_{n,2}$
\begin{align*}
\sum_{\al\in\mc{P}_{n,2}:i_\al< i-2}f_\al(y_i)&\leq \frac{1}{m_{2j_n+2}}\sum_{E\in \mc{H}_i}\norm[E y_i]\,,
\end{align*}
where $\mc{H}_i=\{\supp f_\gamma\cap \supp y_i: \ \gamma\in \suc (\al), \
i-2>i_\al\}$.  Notice that each $\mc{H}_i$ is
$\mc{S}_{n_{2j_n+1}+n_{2j_{i-2}}}$-allowable, thus also
$\mc{S}_{n_{2j_{i}-2}}$-allowable, 
hence by Lemma \ref{f3},\eqref{f3a} $\sum_{E\in\mc{H}_i}\norm[Ey_i]\leq 14\cdot 28$ for each
$i\in I_n$. 
Therefore putting together the above estimates we obtain
\begin{align*}
f|_{\supp \mc{P}_{n,2}}(u_n)&\leq m_{2j_n+1}\sum_{i\in
I_n}b_i\frac{28}{m_{2j_i}}+\frac{6\cdot 28}{m_{2j_n+1}}+m_{2j_n+1}\sum_{i\in
I_n}b_i\frac{14\cdot 28}{m_{2j_n+2}}\leq\frac{1}{m_{2j_{n}}}
\end{align*}
as  $\sum_ib_i=1$. 
Thus erasing the set  $\mc{P}_{n,2}$ we finish the proof
of the 2nd reduction.

\

\textbf{3rd reduction}. For any $n\in\N$ with error $m_{2j_n}^{-1}$ we can
assume that for any $\al\in \mc{P}_{n}$ the special functional $f_\al|_{\ran
u_n}$ is defined by the sequence $(\ran y_{i,k})_{k\in K_i,i\in I_n}$, 
in particular $f_\al|_{\ran u_n}=\frac{1}{m_{2j_n+1}}\sum_{\supp f_\al\cap \supp
y_i\neq\emptyset}f_i^\al$, with  $f_i^\al=\frac{1}{m_{2j_i}}\sum_{k\in
K_i}f^\al_{i,k}$ and $\supp f_{i,k}^\al\subset\ran y_{i,k}$ for each $k\in K_i$,
$i\in I_n$.

\textit{Proof}. Recall that by definition of a dependent sequence 
for any $\beta\in\suc(\mc{P}_n)$ we have $w(\fb)=m_{2s}^{-1}$ for some $s\in\N$.
Fix $i\in I_n$ and let
\begin{align*}
\mc{P}_{n,i}&=\{\beta\in \suc(\mc{P}_n):\ w(\fb)< m_{2j_{i}}^{-1}\},\\
 \mc{Q}_{n,i}&=\{\beta\in \suc(\mc{P}_n):\ w(\fb)> m_{2j_{i}-2}^{-1}\}.
\end{align*}
Notice that by 1st and 2nd reductions the family $(\fb)_{\beta\in\mc{P}_{n,i}}$
is $\mc{S}_{\frac{6}{5}n_{2j_n+1}}$-allowable, thus also
$\mc{S}_{n_{2j_i}-1}$-allowable. Moreover $y_i$ is a scaled
$(n_{2j_i},m_{2j_i}^{-2})$-scc. 
Therefore we can repeat the reasoning from the 2nd reduction obtaining
$f|_{\supp\mc{P}_{n,i}}(y_i)\leq 2m_{2j_{i}-1}^{-1}$.

By the 1st and 2nd reductions the family $\{f_\gamma:\ \gamma\in\suc(\beta):\beta\in\mc{Q}_{n,i}\}$ is $\mc{S}_{\frac{6}{5}n_{2j_n+1}+n_{2j_i-4}}$-allowable, thus $\mc{S}_{n_{2j_i-2}}$-allowable. Thus by Lemma \ref{f3}\eqref{f3a}, we have
\begin{align*}
 f|_{\supp\mc{Q}_{n,i}}(y_i)&\leq \frac{1}{m_{2j_n+1}}\sum_{\beta\in \mc{Q}_{n,i}}w(f_\beta)\sum_{\gamma\in\suc(\beta)}f_\gamma(y_i)\leq \frac{14\cdot 28}{m_{2j_n+1}^{3}},
\end{align*}
as any $f_\beta$ is an immediate descendant of some special functional $f_\al$, $\al\in\mc{P}_n$, with $w(f_\al)=m_{2j_n+1}^{-1}$ by the 2nd reduction. 

Let $\tilde{\mc{P}}_{n,3}=\cup_i(\mc{P}_{n,i}\cap\supp y_i)$ and $\tilde{\mc{Q}}_{n,3}=\cup_i(\mc{Q}_{n,i}\cap\supp y_i)$. Then by the above
\begin{align*}
 f_{\tilde{\mc{P}}_{n,3}\cup\tilde{\mc{Q}}_{n,3}}(u_n)&
 \leq m_{2j_n+1}\sum_{i\in I_n}\frac{2b_i}{m_{2j_{i}-1}}+14\cdot
28m_{2j_n+1}\sum_{i\in I_n}\frac{b_i}{m_{2j_n+1}^{3}}\leq \frac{1}{m_{2j_{n}+1}}.
\end{align*}
Now notice that for any $\al\in\mc{P}_n$ there is at most one $i\in I_n$ with $w(\fb)=m_{2j_i-2}^{-1}$ for some $\beta\in\suc(\al)$. Denote the set of such $i$'s by $K_n$.  Therefore, by the 1st reduction, as $u_n$ is a scaled $(n_{2j_n+1}, m_{2j_n+1}^{-1})$-scc, 
$$
m_{2j_n+1}\sum_{i\in I_n}b_i\sum_{\beta\in\suc(\mc{P}_{n}): w(\fb)=m_{2j_i-2}^{-1}}\fb(y_i)\leq m_{2j_n+1}\norm[\sum_{i\in K_n}b_iy_i]\leq \frac{2\cdot 28}{m_{2j_n+1}}.
$$
Summing up we obtain that with error  $57m_{2j_{n}+1}^{-1}$
we can assume that for any $n\in\N$ and $\al\in
\mc{P}_{n}$ we have $w(f_\beta)=m_{2j_i}^{-1}$ for any $\beta\in\suc \al$ with
$\supp f_\beta\cap\supp y_i\neq\emptyset$. In particular it follows that $\supp
f_\beta$, $\beta\in\suc(\mc{P}_n)$,  intersects at most one of $\supp y_i$'s. 

By (2) in Def. \ref{special} and (D) it follows that $f|_{[1,\dots,\maxsupp u_n]}$ is a special functional defined by the intervals $(\ran y_{i,k})_{k\in K_i,i\in I_n}$. In order to obtain that suitable $f_{i,k}^\al$ satisfy $\supp f_{i,k}^\al\subset\ran y_{i,k}$ we shall make one more correction. 

For any $\al\in \mc{P}_{n}$ let $i_\al=\max\{i\in I_n:\ \supp f_\al\cap \supp y_i\neq\emptyset\}$. Put $F_n=\{i_\al:\ \al\in \mc{P}_{n}\}$. Notice that, as $u_n$ is a scaled $(n_{2j_n+1},m_{2j_n+1}^{-2})$-scc and $(f_\al|_{\ran u_n})_{\al\in \mc{P}_{n}}$ is $\mc{S}_{n_{2j_n+1}-1}$-allowable by the 1st reduction, we obtain 
$$
f(m_{2j_n+1}\sum_{i\in F_n}b_iy_i)\leq \norm[m_{2j_n+1}\sum_{i\in F_n}b_iy_i]\leq\frac{2}{m_{2j_n+1}}.
$$
In case $w(f_{i}^{\al})=m_{2j_{i}}^{-1}$ for some $i>\min I_{n}$ it follows
that  $\ran(f_{\min I_{n},k})\subset \ran(y_{\min I_{n},k})$ for
 every $k.$  Otherwise deleting part of $y_{\min I_{n}}$, with error $m_{2j_{n}+1}^{-1}$ we may assume
also that $\ran(f_{\min I_{n},k})\subset \ran(y_{\min I_{n},k})$. 
Therefore, after erasing $m_{2j_n+1}\sum_{i\in F_n}b_iy_i$ with error $2m_{2j_n+1}^{-1}$, by (2) and (3) of Def. \ref{special} we finish the proof of the 3rd reduction. 

\

\textbf{4th reduction}. Given $n\in\N$ let 
$$
\mc{D}_n=\{\xi\in \mt: \xi\prec\al \text{ for some }\al\in\mc{P}_n \text{ and }f_\xi \text{ is a special functional}\}.
$$ 
For $\xi\in \mc{D}_{n}$ let  $(E_j^\xi)_{j\in A_l,l\in G_s,s\in F}$ be the sequence of the intervals that determines $f_{\xi}$ (see  Def. \ref{special}).  
With error $m_{2j_n}^{-1}$ we can assume that for any $i\in I_n$, $k\in K_i$ and any $\xi\in\mc{D}_n$ with $f_\xi(y_{i,k})\neq 0$ there are $j\in A_l$, $l\in G_s$, $s\in F$ such that $\ran y_{i,k}\subset E_j^\xi$.

\textit{Proof}.
For any $\xi\in\mc{D}_n$ let $i_\xi\in I_n$ be minimal with $\supp f_\xi\cap\ran y_{i_\xi}\neq\emptyset$.

First notice that the family $(y_{i_\xi})_{\xi\in\mc{D}_n}$ is $\mc{S}_{n_{2j_n+1}-1}$-admissible. Indeed, for any $\xi\in\mc{D}_n$ we have $\supp f_{\xi}\cap\ran u_n=\cup\{\supp \fa: \xi\prec\al\in \mc{P}_n\}\cap \ran u_n$ by the 1st reduction, thus  
\begin{equation}\label{admis-1}
(\minsupp f_\xi|_{\ran u_n})_{\xi\in \mc{D}_n}\subset (\minsupp f_\al|_{\ran u_n})_{\al\in\mc{P}_n}\, .
\end{equation}
Hence by the 1st reduction the family $(y_{i_\xi})_{\xi\in\mc{D}_{n}\setminus \{\xi_0\}}$, where $f_{\xi_0}|_{\ran u_n}$ has the smallest $\minsupp$ among $f_\xi|_{\ran u_n}$, $\xi\in\mc{D}_n$, is $\mc{S}_{\frac{1}{5}n_{2j_n+1}}$-admissible, which yields the desired observation. 

Therefore, as $u_n$ is a scaled $(n_{2j_n+1}, m_{2j_n+1}^{-1})$-scc,  we obtain that
$$
\norm[m_{2j_n+1}\sum_{\xi\in \mc{D}_n}b_{i_\xi}y_{i_\xi}]\leq \frac{2\cdot 28}{m_{2j_n+1}}\,.
$$
Thus with the above error we may assume that  for all $\xi\in \mc{D}_{n}$ there is some  $i>i_{\xi}$ with $f_{\xi}(y_{i})\ne 0$.
Let now $\overline{E}_j^\xi=E_j^\xi\cap (\maxsupp y_{i_\xi},\maxsupp u_n]$  for any $\xi\in \mc{D}_n$ and element $E_j^\xi$ of a sequence defining $f_\xi$. 
It follows that for any $\xi\in\mc{D}_n$ we have 
\begin{equation}\label{admis0}
\minsupp f_\xi|_{\ran u_n}\leq \min \bigcup_{j\in A_l,l\in G_s,s\in F}\overline{E}^\xi_j\,.
\end{equation}
Notice that for any $\al\in\mc{P}_n$ there can be at most $m_{2j_n+1}$ many $\xi\in\mc{D}_n$ with  $\minsupp f_\al|_{\ran u_n}=\minsupp f_\xi|_{\ran u_n}$, as this relation implies that $\xi\prec\al$ and $ord(\al)\leq m_{2j_n+1}$ by the 1st reduction. As $\minsupp u_n>m_{2j_n+1}$, by \eqref{admis0} for any $\al\in\mc{P}_n$ we have
\begin{equation}\label{admis0.5}
\{\min \bigcup_{j\in A_l,l\in G_s,s\in F}\overline{E}^\xi_j: \ \minsupp f_\xi|_{\ran u_n}=\minsupp f_\al|_{\ran u_n}\}\in\mc{S}_{1}.
\end{equation}
Therefore by \eqref{admis-1}, \eqref{admis0.5}, the 1st reduction and Lemma \ref{schreier} we obtain that
\begin{equation}\label{admis1}
\{\min \bigcup_{j\in A_{l},l\in G_s,s\in F}\overline{E}^{\xi}_{j}:\ \xi\in\mc{D}_n\} \in\mc{S}_{\frac{1}{5}n_{2j_n+1}+1}^M=\mc{S}_{\frac{1}{5}n_{2j_n+1}+1}.
\end{equation}
On the other hand by Def. \ref{special} for any $(E_j^\xi)_{j\in A_l,l\in G_s, s\in F}$ with $\xi\in \mc{D}_n$, any sum $\cup_{j\in A_l}E_j^\xi$ contains $\supp f_\gamma$ for some $\gamma\in\suc(\suc \xi)$ provided $f_\xi|_{\cup_{j\in A_l}E_j^\xi}\neq 0$. Denote the set of all such $l$'s by $\overline{G}_s$. 
We shall prove that with the declared error for any interval $E_j^\xi$, for some $j\in A_l$, $l\in \overline{G}_s$, $s\in F$ defining some $\xi\in \mc{D}_n$ and any $i\in I_n$, $k\in K_i$ we have either $\ran y_{i,k}\cap E_j^\xi=\emptyset$ or $\ran y_{i,k}\subset E_j^\xi$, which ends the proof of the 4th reduction. 

Take any $l\in \overline{G}_s$, $s\in F$ attached to some $\xi\in \mc{D}_n$ and consider $\gamma\in\suc(\suc \xi)$ with $\cup_{j\in A_l}E_j^\xi\supset\supp f_\gamma$. Since $\ran u_n\subset \ran\mc{P}_n$ by the 1st reduction and by definition of $\mc{P}_n$ for such $\gamma$ we have either $\gamma\preceq \al$ for some $\al\in\mc{P}_n$ or $\ran f_\gamma\cap \ran u_n=\emptyset$. Therefore by the 1st reduction and as functionals $(f_\gamma)_{\gamma\in\suc(\suc\xi)}$ have successive supports,  for any $\xi\in\mc{D}_n$ we have
\begin{equation}\label{admis2}
\{\min(\bigcup_{j\in A_l}E_j^\xi\cap\ran u_n):\ l\in \overline{G}_s,s\in F\}\in\mc{S}_{\frac{1}{5}n_{2j_n+1}+1} 
\end{equation}
Putting together \eqref{admis1} and \eqref{admis2}, by Lemma \ref{schreier} we obtain that 
\begin{equation}\label{admis3}
\{\min(\bigcup_{j\in A_l}\overline{E}_j^\xi\cap\ran u_n):\ l\in
\overline{G}_s,s\in F, \xi\in \mc{D}_n\}
\in
\mc{S}_{\frac{2}{5}n_{2j_n+1}+2} 
\end{equation}
Set
\begin{align*}
 J_n=\{\min I_n\}\cup\{i\in I_{n}:\min\bigcup_{j\in A_l}\overline{E}^{\xi}_j\in(& \maxsupp y_{i-1},\maxsupp y_i]
 \\
 &\text{ for some } l\in \overline{G}_s,s\in F, \xi\in \mc{D}_{n}\}\,.
\end{align*} 
Using \eqref{admis3}, as $u_n$ is a scaled $(n_{2j_n+1}, m_{2j_n+1}^{-1})$-scc, we obtain that 
$$
\norm[m_{2j_n+1}\sum_{i\in J_n}b_iy_i]\leq \frac{2\cdot 28}{m_{2j_n+1}}\,.
$$ 
Thus for any $i\in I_n\setminus J_n$ and $(E_j^\xi)_{j\in A_l}$ with $l\in \overline{G}_s,s\in F$, $\xi\in \mc{D}_n$, if $\bigcup_{j\in A_l}\overline{E}_j^\xi\cap\ran y_i\neq\emptyset$, then $\min \bigcup_{j\in A_l}\overline{E}_j^\xi\leq \maxsupp y_{i-1}$. By (4) in Def. \ref{special} it follows that the family $(\overline{E}_j^\xi\cap\ran y_i)_{j\in A_l}$ is $\mc{S}_{\rho(\maxsupp y_{i-1})}$-admissible and consequently by \eqref{admis3} and Lemma \ref{schreier} for any $i\in I_n$ we have
$$
\{\min (\overline{E}_j^\xi\cap\ran y_i):\ j\in A_l, l\in \overline{G}_s,s\in F, \xi\in\mc{D}_n\}\in
\mc{S}_{\rho(\maxsupp y_{i-1})+n_{2j_n+1}}
$$ 
As $y_i$ is a scaled $(n_{2j_i}, m_{2j_i}^{-1})$-scc,  by the condition (E) we obtain that 
$$
\norm[m_{2j_i}\sum_{k\in L_i}b_{i,k}y_{i,k}]\leq \frac{2\cdot 2}{m_{2j_i}},
$$ 
where $L_i$ denotes the set of all $k\in K_i$ such that $\min \overline{E}_j^\xi$ or $\max \overline{E}_j^\xi$ belongs to the interval $(\maxsupp y_{i,k-1},\maxsupp y_{i,k}]$ (in case $k>\min K_i$) or to the interval $(\maxsupp y_{i-1,\max K_{i-1}},\maxsupp y_{i,\min K_i}]$  (in case $k=\min K_i$)  for some element $E_j^\xi$ of a sequence $(E_j^\xi)_{j\in A_l,l\in \overline{G}_s,s\in F}$ defining $f_\xi$ for some $\xi\in \mc{D}_{n}$. It follows that 
$$
\norm[m_{2j_n+1}\sum_{i\in I_n\setminus J_n}b_im_{2j_i}\sum_{k\in L_i}b_{i,k}y_{i,k}]\leq m_{2j_n+1}\sum_{i\in I_n}\frac{4b_i}{m_{2j_i}}\leq \frac{1}{m_{2j_n+1}}\,.
$$
As $\min E_j^\xi,\max E_j^\xi\in \{\min \overline{E}_j^\xi, \max\overline{E}_j^\xi\}\cup [1,\maxsupp y_{i_\xi}]\cup (\maxsupp u_n,\infty)$, 
after erasing $(y_{i_\xi})_{\xi\in\mc{D}_n}$, $(y_i)_{i\in J_n}$ and
$(y_{i,k})_{k\in L_i,i\in\ I_n}$ 
with error $113m_{2j_n+1}^{-1}< m_{2j_n}^{-1}$ we obtain that for any interval $E_j^\xi$, with $j\in A_l$, $l\in \overline{G}_s$, $s\in F$ defining some $\xi\in \mc{D}_n$ and any $i\in I_n$, $k\in K_i$ we have either $\ran y_{i,k}\cap E_j^\xi=\emptyset$ or $\ran y_{i,k}\subset E_j^\xi$, which proves the 4th reduction.

\

The total error we paid for reductions is $\sum_n5m_{2j_n}^{-1}\leq 1/2$ by (A).

\textbf{Replacement}. By $\tilde{f}$ denote the restriction of $f$ obtained by the above reduction. By the above $\tilde{f}(\sum_na_nu_n)\geq 1/2$.

Fix now $i\in I_n, k\in J_i$ and denote by $\Gamma_{i,k}$ the collection of all $\gamma\in\mt$ with $\gamma\in\suc(\suc(\al))$, for some $\al\in \mc{P}_n$, with $\supp f_\gamma\subset\ran y_{i,k}$. By the 3rd reduction $\supp\Gamma_{i,k}\supset\supp y_{i,k}\cap\supp f$. We pick $\gamma_{i,k}\in\Gamma_{i,k}$ with the biggest $tag(\gamma_{i,k})$, erase all other $f_\gamma$ with $\gamma\in\Gamma_{i,k}$ and replace $f_{\gamma_{i,k}}$ by $z_{i,k}^*$. Denote the new functional defined by the modified tree by $g$. 

Notice first that the replacement is correct, i.e. $g\in W$, since the change does not affect the sequences $(E_j)_j=(\ran y_{i,k})_{i,k}$ defining special functionals $f_\al$, $\al\in\mc{P}_n$, nor any other sequence in $\mc{D}_n$ by the 4th reduction. Indeed, assume $\gamma_{i,k}\succ\xi$ for some $\xi\in\mc{D}_n$. Then $\supp f_\xi\cap \ran y_{i,k}\neq\emptyset$ thus by the 4th reduction $\ran z_{i,k}^*=\ran y_{i,k}\subset E_j^\xi$ for any element $E_j^\xi$ of a sequence defining $f_\xi$. 
 
Notice also that for $\gamma_1\neq\gamma_2$ with
$\gamma_1,\gamma_2\in\Gamma_{i,k}$,  $\gamma_1\in\suc(\beta_1)$ and
$\gamma_2\in\suc(\beta_2)$ we have, by definition of a special
functional and the 3rd reduction, that $\beta_1,\beta_2$ are incomparable. Therefore $(f_\gamma)_{\gamma\in\Gamma_{i,k}}$ is $\mc{S}_{n_{2j_n+1}-1+n_{2j_i}}$-allowable, and using Lemma \ref{1.4} we obtain 
$$ 
\tilde{f}(y_{i,k})=\sum_{\gamma\in\Gamma_{i,k}}tag(\gamma)f_\gamma(y_{i,k})\leq 3\ tag(\gamma_{i,k})=3\ tag(\gamma_{i,k}) z^*_{i,k}(z_{i,k})
$$
Thus we have $1/2\leq\tilde{f}(\sum_{n}a_nu_n)\leq 3g(\sum_{n}a_nv_n)$, which ends the proof. 
\end{proof}
\begin{definition}\cite{fr1} A Banach space with a basis $(e_n)$ is called \textit{tight by range} if for any block subspace $Y$ of $X$ spanned by a block sequence $(y_n)$, $Y$ does not embed into $[e_i:\ i\not\in\cup_n\ran y_n]$. 
\end{definition}
It was shown in \cite{fr1} that a Banach space is tight by range iff 
any its two block subspaces with disjoint ranges are incomparable.
\begin{theorem}
 The space $\mc{X}_{cr}$ is tight by range. 
\end{theorem}
\begin{proof}
Let $(z_r)_r$ be a block sequence. We show that there exists  no bounded operator  $T$ such that  $\supp T(z_r)\cap\ran(z_r)=\emptyset$ and  $T$ can be extended  to an isomorphism from $[z_r: r\in\N]$  to $X$, which will prove that $\mc{X}_{cr}$ is tight by range. By standard arguments we may assume that $\norm[T]\leq 1$,  $(Tz_r)_j$ is a block sequence and $\ran(z_r+Tz_r)<\ran(z_{r+1}+Tz_{r+1})$ for every $r\in\N$. Passining to further subsequence  we may assume that either
$\minsupp z_{r}<\minsupp Tz_{r}$ for all $r$ or $\minsupp z_{r}>\minsupp Tz_{r}$ for all $r$. Notice that if  $\sum_ra_{r}z_{r}$ is an $(n,\e)$-scc, then in the first case $\sum_{r}a_{r}Tz_{r}$ is also an $(n,\e)$-scc, while in the second $\sum_{r}a_{r}Tz_{r}$ is an $(n,\e)$-scc up to the first element. With this observation we can adapt here the argument of the proof of quasi-minimality of $\mc{X}_{cr}$. 

For any fixed $j\in\N$ we construct a special sequence $(x^*_{i})_{i\in F}$ and a block sequence $(x_{i})_{i\in F}\subset [z_r: r\in\N]$ such that
\begin{enumerate}
\item[(A')] $x=m_{2j+1}\sum_{i\in F}b_ix_i$ is a scaled $(n_{2j+1},m_{2j+1}^{-2})$-scc of 28-RIS $(x_i)_{i\in F}$,\\ $x^*=m_{2j+1}^{-1}\sum_{i\in F}x_i^*$ is a special functional in $W$,
\item[(B')] $x_i=m_{2j_{i}}\sum_{k\in K_{i}}b_{i,k}x_{i,k}$ is a scaled $(n_{2j_i},m_{2j_{i}}^{-3})$-scc of 2-RIS $(x_{i,k})$, \\
$x^*_i=m_{2j_{i}}^{-1}\sum_{k\in K_{i}}x^*_{i,k}$ for any $i\in F$,
\item[(C')] $x_{i,k}$ is a normalized $(n_{2j_{i,k}}, m_{2j_{i,k}}^{-2})$-scc supported on $(z_{r})_{r\in A_k}$, \\ $x^*_{i,k}(x_{i,k})=1$, $\ran x_{i,k}=\ran x_{i,k}^*$ for any $k\in K_i$, $i\in F$,
\item[(D')] $(x^*_{i})_{i\in F}$ is a $(2j+1)$-dependent sequence defined by $(E_r)_{r\in \tilde{A}_{k},k\in K_i, i\in F}$, where $\tilde{A}_k=A_k$ and 
$E_r=\ran z_r$ for any $r\in \tilde{A}_k$, $k\in K_i\setminus \{\min K_i\}$, $i\in F$,\\ 
and $\tilde{A}_{\min K_i}$ is a singleton indexing the interval $\ran x_{i,\min K_i}$ for each $i\in F$,
\item[(E')] $A_{k}\in\mc{S}_{\rho(\maxsupp x_{i,k-1})}$ for any $k\in
  K_i\setminus\{\min K_i\}$, $i\in F$.
\end{enumerate}
Notice that condition (E') ensures that in (D') we have a correctly defined depended sequence. We can ensure conditions (B'), (C') and (E') by Lemma \ref{f1} and definition of the function $\rho$. Indeed, having chosen $x_{i,k-1}$ for $i\in F$ and $k\in K_i\setminus\{\min K_i\}$ we are able to choose the next element $x_{i,k}$ of a RIS with weight $m_{j_{i,k}}$ satisfying $\maxsupp x_{i,k-1}\leq m_{j_{i,k}}/m_{j_{i,k-1}}$ and supported on $[z_r:r\in A_k]$ for some  $A_k\in\mc{S}_{\rho(\maxsupp x_{i,k-1})}$ by definition of $\rho$ and the condition \eqref{rho}, as $\rho(\maxsupp x_{i,k-1})$ enables to choose an $(n_{2s},m_{2s}^{-2})$-scc with weight $m_{2s}\geq \maxsupp x_{i,k-1}^2\geq \maxsupp x_{i,k-1}m_{j_{i,k-1}}$.   

Notice that the construction of $x$ differs from the choice of vectors $u_n$ in one aspect - in the speed of growth of $(m_{2j_i})$. In the previous case we demanded high speed of growth in condition (E), here we tame the speed of growth $(m_{2j_i})$ as much as possible, in order to obtain condition (E') and in consequence to be able to use $(\ran z_j)_j$ as the intervals defining special functional. Recall that in previous case we took as intervals defining special functionals the sets $(\ran y_{i,k})_{i,k}$, so we used vectors on "higher" level. Again by $(\beta)$ we can assume that all coefficients of $f$ and $x$ are non-negative.

Notice that by (A') we have $\norm[x]\geq x^*(x)\geq 1$. In order to estimate $\norm[Tx]$ we take $f\in W$ with a tree-analysis and repeat 1st, 2nd and 3rd reductions from the proof of the previous theorem for one vector $x$ instead of linear combination of $(u_n)$. Recall that condition (E) was required only in the last reduction we shall not repeat here. Within 1st, 2nd and 3rd reductions we define as before 
\begin{align*}
\mc{P}=\{\al\in\mt: \ &\supp f_\al\cap\supp Tx\neq\emptyset\text{ and }\\ &\al \text{  is minimal in }\mt \text{ with }w(f_\al)\leq m_{2j+1}^{-1}\}
\end{align*}
obtaining after reductions that there is a restriction $\tilde{f}$ of the functional $f$ such that 

\noindent (1) $f(Tx)\leq\tilde{f}(Tx)+4m_{2j}^{-1}$ ,

\noindent (2) $\supp \tilde{f}\subset\supp\tilde{\mc{P}}$, where 
\begin{align*}
\tilde{\mc{P}}=\{\al\in\mt: \ &\supp f_\al\cap\supp Tx\neq\emptyset\text{ and }\\ & \al\text{  is minimal in }\mt \text{ with } w(f_\al)= m_{2j+1}^{-1}\},
\end{align*}
\noindent (3) for any $\al\in\tilde{\mc{P}}$ the special functional $f_\al|_{\ran Tx}$ is defined by the sequence $(E_r)_{r\in \tilde{A}_{k}, k\in K_i,i\in F}$, in particular $f_\al|_{\ran Tx}=\frac{1}{m_{2j+1}}\sum_{\supp f_\al\cap \supp Tx_i\neq\emptyset}f_i^\al$, with  $f_i^\al=\frac{1}{m_{2j_i}}\sum_{k\in K_i}f^\al_{i,k}$ and $\supp f_{i,k}^\al\subset\cup_{j\in A_{k}}\ran z_j$ for each $k\in K_i\setminus \{\min K_i\}$, $i\in F$.

Therefore, as $\supp T(z_j)\cap\ran z_j=\emptyset$, we obtain 
$$
\tilde{f}(Tx_i)\leq \sum_{\al\in \tilde{\mc{P}}} \frac{w(f^\al_i)}{m_{2j_i}}f^\al_{i,\min K_i}(Tx_i)\leq \frac{16}{m_{2j_i}}
$$
as $\sum_{\al\in \tilde{\mc{P}}} w(f^\al_i)f^\al_{i,\min K_i}$ is a norming functional obtained from $\tilde{f}$ by replacing in its tree-analysis each $f^\al_i$ by $f^\al_{i,\min K_i}$. Finally we have 
$$
\norm[Tx]\leq \frac{3}{m_{2j}}+\sum_{i\in F}\frac{16}{m_{2j_i}}\leq \frac{4}{m_{2j}}\,.
$$
Since $j\in\N$ is arbitrarily large, the above shows that $T$ is not an isomorphism onto image. 
\end{proof}
\begin{remark} 
Consider a Banach space $\mc{Y}$ satisfying conditions
$(\al)$-$(\delta)$ with respect to $(\mc{A}_n)_{n\in\N}$-admissible sets
instead of  Schreier admissible or allowable sets. Then by repeating the reasoning above we
obtain another example of a Banach space with unconditional basis,
which is quasi-minimal and tight by range, as in \cite{amp}.  As in the
example of \cite{amp} the space $\mc{Y}$ is also locally minimal, as
saturated with $\ell_\infty^n$'s (see the next section).
\end{remark}
\section{Local minimality of the space $\mc{X}_{(4)}$ of \cite{amp}}
We recall first briefly the construction of the norming set $W_4$ of the space $\mc{X}_{(4)}$ constructed in \cite{amp}. We fix two sequences of natural numbers $(m_{j})_{j}$ and $(n_{j})_{j}$ and a partition of $\N$ into two infinite sets $L_{1},L_{2}$ as in the definition of $W$ in Section 2. Let 
$W_4$ be the smallest subset of $c_{00}(\N)$ satisfying the following
\begin{enumerate}
\item $(\pm e_n)_n\in W_4$, where $(e_n)_n$ is the canonical basis of $c_{00}(\N)$,
\item for any $f\in W_4$ and $g\in c_{00}(\N)$ with $\vvert[f]=\vvert[g]$ also $g\in W_4$,
\item $W_4$ is closed under the projection on intervals  of $\N$,
\item $W_4$ is closed  under the $(\mc{A}_{n_{2j}},m_{2j}^{-1})$-ope\-ra\-tions on any block se\-quen\-ces,
\item $W_4$ is closed under the  $(\mc{A}_{n_{2j+1}},m_{2j+1}^{-1})$-opera\-tions on $(2j+1)$-special sequences,
\item  $W_4$ is closed under the  $G$-operation, defined as follows. For any set $F=\{n_{1}<\dots<n_{2q}\}\subset\N$ which is Schreier (i.e. $2q\leq n_1$) we set
$$
S_{F}f=\chi_{\cup_{p=1}^{q}[n_{2p-1}, n_{2p})}f.
$$
The $G$-operation associates with any $f\in c_{00}$ the vector $g=\frac{1}{2}S_{F}f$, for any $F$ as above.
\end{enumerate}
In order to complete the definition we define special sequences. A  sequence $f_{1}<f_{2}<\dots<f_{n_{2j+1}}$ in $W_4$ is a $(2j+1)$-\textit{special sequence}, if the following are satisfied
\begin{enumerate}
\item  for every $i=1,\dots,n_{2j+1}$, $w(f_{i})=m_{2j_{i}}$ where   $j_{1}\in L_{1}$, $j_i\in L_2$ for any $i>1$ and
$n_{2j+1}<m_{2j_{1}}<\dots<m_{2j_{n_{2j+1}}}$,
\item $m_{2j_{i+1}}>(\maxsupp f_{i})m_{2j_{i}} $ for any $1\leq  i< n_{2j+1}$,
\item for $1< i\leq n_{2j+1}$ the sequence $(\vvert[f_1],\vvert[f_2],\dots,\vvert[f_{i-1}])$ is uniquely determined by $w(f_{i})$.
\end{enumerate}
Notice that the norming set $K$ of the mixed Tsirelson space $T[(\mc{A}_{n_j},\frac{1}{m_j} )_j]$ is closed under the projections on subsets of $\N$ and $m_1=2$. It follows that $W_4\subset K$. This observation together with unconditionality of the basis in $\mc{X}_{(4)}$ allows for repeating in the space $\mc{X}_{(4)}$ the argument of \cite{m} that $\ell_\infty$ is finitely disjointly representable in every infinitely dimensional subspace of $T[(\mc{A}_{n_j}, \frac{1}{m_j})_j]$. The quoted reasoning uses only the estimation of the action of any functional $f\in K$ on a linear combination of some block sequence by action of another functional $g\in K$ on an analogous combination of the basis $(e_n)$ by means of modifying the tree-analysis of $f$ into the tree-analysis of $g$. As $W_4\subset K$ we can adapt the mentioned reasoning of \cite{m} in the space $\mc{X}_{(4)}$, obtaining the following theorem, which answers the question (2) of \cite{fr1}.
\begin{theorem}
The space $\mc{X}_{(4)}$ is locally minimal, i.e. $\mc{X}_{(4)}$ is finitely represented in any of its infinitely dimensional subspaces. 
\end{theorem}
\subsection*{Acknowledgements}
This research was partly   supported by the program  ``$A\rho\iota\sigma\tau\epsilon \iota\alpha"$

\end{document}